\documentclass[12pt,reqno]{amsart}

\usepackage{amssymb}

\newtheorem{theorem}{Theorem}[section]
\newtheorem{proposition}[theorem]{Proposition}
\newtheorem{lemma}[theorem]{Lemma}

\theoremstyle{definition}
\newtheorem{remark}[theorem]{Remark}

\numberwithin{equation}{section}

\begin{document}

\baselineskip=15pt

\title[Vanishing theorem for co-Higgs bundles on the moduli space of bundles]{A 
vanishing theorem for co-Higgs bundles on the moduli space of bundles}

\author[I. Biswas]{Indranil Biswas}

\address{School of Mathematics, Tata Institute of Fundamental
Research, Homi Bhabha Road, Bombay 400005, India}

\email{indranil@math.tifr.res.in}

\author[S. Rayan]{Steven Rayan}

\address{Department of Mathematics \& Statistics, University of Saskatchewan, McLean 
Hall, 106 Wiggins Road, Saskatoon, SK, S7N 5E6, Canada}

\email{rayan@math.usask.ca}



\date{}

\begin{abstract}
We consider smooth moduli spaces of semistable vector bundles of fixed rank and 
determinant on a compact Riemann surface $X$ of genus at least $3$. The choice of a 
Poincar\'e bundle for such a moduli space $M$ induces an isomorphism between $X$ and 
a component of the moduli space of semistable sheaves over $M$. We prove that
$\dim H^0(M,\, \text{End}({\mathcal E})\otimes TM)\,=\, 1$ for
any vector bundle $\mathcal E$ on $M$ coming from this component. Furthermore, there are
no nonzero integrable co-Higgs fields on $\mathcal E$.
\end{abstract}

\maketitle

\section{Introduction}

Higgs bundles have been mostly studied on Riemann surfaces although moduli spaces 
have been constructed for arbitrary smooth projective varieties \cite{Si}. On 
Riemann surfaces, Higgs bundles moduli spaces are only nonempty in positive genus, 
where the canonical line bundle has sections. One way to extend the theory of Higgs 
bundles to genus $0$, without introducing a parabolic structure, is to consider 
\emph{co-Higgs bundles} as in \cite{R1}. In general, a co-Higgs bundle on a variety 
$X$ is a holomorphic bundle $E$ with a co-Higgs field $\phi\,:\,E\,
\longrightarrow\, E\otimes TX$ for 
which $\phi\bigwedge\phi$ vanishes in $H^0(X,\, \mbox{End}(E)\otimes\bigwedge^2TX)$. This 
vanishing is the analogue of the integrability condition in \cite{Si}. Co-Higgs 
bundles arose originally in generalized complex geometry \cite{H3}, as the limit of 
generalized holomorphic bundles as a generalized complex structure becomes ordinary 
complex.

The pattern continues in higher dimension, with Higgs bundles and co-Higgs bundles 
existing largely as general-type and Fano phenomena, respectively (see 
\cite{R2,Co,B1,B2,BB}). Families of integrable co-Higgs bundles on Fano surfaces, 
namely $\mathbb P^2$ and $\mathbb P^1\times\mathbb P^1$, have been constructed in 
\cite{R2,VC}. Another natural variety to consider is $M_\xi$, the moduli space of 
semistable bundles of rank $r$ and fixed determinant $\xi$ on a compact Riemann surface $X$ 
of genus $g$ with $g\, \geq\, 2$. This variety is known to be Fano and has 
$\mbox{Pic}(M_\xi)\,\cong\,\mathbb Z$. It is an irreducible normal projective variety,
and it is smooth whenever $\gcd(r,{\rm degree}(\xi))\,=\,1$.

One immediate example of an integrable co-Higgs bundle on $M_\xi$ is the one with 
$E\,=\,{\mathcal O}_{M_\xi}\oplus TM_\xi$ and the co-Higgs field $\phi$ that maps $TM_\xi$ 
to $\mathcal O_{M_\xi}\otimes TM_\xi$ via the identity and acts as zero on the
component ${\mathcal O}_{M_\xi}\, \subset\, E$. This
co-Higgs field is nilpotent of order $2$, and so in particular 
$\phi\bigwedge\phi\,=\,0$. This is the so-called canonical co-Higgs bundle, which can be 
defined on any variety \cite{R0,H3}. A sufficient condition for $(E,\,\phi)$ to be 
slope stable in the sense of Hitchin \cite{H1} is that $TM_\xi$ should be slope stable as 
a bundle \cite{Mi,R0}. This has been conjectural for some time. It has been known 
for $r\,=\,2$ since \cite{Hw} (see \cite{Iy} for $r\,>\,2$). It is natural to ask if
there are examples of 
integrable co-Higgs bundles on $M_\xi$ where $E$ is stable as a vector bundle, as 
the stability of $E$ makes the stability of $(E,\,\phi)$ automatic. Under
the assumption that $g\, \geq\, 3$, we show 
that if $(E,\,\phi)$ is integrable, and $E$ is a stable 
bundle from the component of the moduli space of sheaves on $M_\xi$ isomorphic to 
$X$, then $\phi\,=\,0$.

More precisely, we prove that the total space of the tangent bundle $TX$ is a 
component of the moduli space of semistable co-Higgs bundles on $M_\xi$, while the 
moduli space of integrable co-Higgs bundles sits inside $TX$ as the image of the 
zero section $X\, \longrightarrow\, TX$.

This work is inspired in part by a conversation between E. Witten and the second 
named author concerning branched covers of moduli spaces of bundles, with reference 
to spectral covers constructed using Higgs bundles with poles \cite{FW}. We observe 
that, according to Theorem 3.1, co-Higgs bundles $(E,\,\phi)$ with $E$ in the given 
component of the space of stable sheaves on $M_\xi$ do not generate interesting 
spectral covers --- they are just the zero section in the total space of $TM_\xi$. In
the canonical co-Higgs example, the spectral cover is the first-order 
neighborhood of the zero section of $TM_\xi$ and can potentially be perturbed
by deforming $(E,\,\phi)$ in such a way that $\phi$ is no longer 
nilpotent but still integrable. We leave these speculations as inspiration for 
future work.

\section{A class of co-Higgs bundles}

\subsection{Poincar\'e bundle and co-Higgs fields}

Let $X$ be a compact connected Riemann surface of genus $g$, with $g\, \geq\, 3$. The
holomorphic tangent and cotangent bundles of $X$ will be denoted by $TX$ and $K_X$
respectively.
Fix an integer $r\, \geq\, 2$ and also a holomorphic line bundle $\xi\, \longrightarrow\, X$
such that $\text{degree}(\xi)$ is coprime to $r$. Let $M_\xi$ denote the moduli space
of stable holomorphic vector bundles $E$ on $X$ such that $\text{rank}(E)\,=\, r$ and
$\bigwedge^r E\,=\, \xi$. This $M_\xi$ is an irreducible smooth projective variety of
dimension $(r^2-1)(g-1)$. The Picard group of $M_\xi$ is isomorphic to $\mathbb Z$. Hence
the notion of (semi)stability of vector bundles over $M_\xi$ does not depend on the choice
of polarization on $M_\xi$.

A Poincar\'e vector bundle on $X\times M_\xi$ is a holomorphic
vector bundle
\begin{equation}\label{e0}
{\mathcal E}\, \longrightarrow\, X\times M_\xi
\end{equation}
such that for every point $E\, \in\, M_\xi$, the restriction ${\mathcal E}\vert_{X\times
\{E\}}$ lies in the isomorphism class of vector bundles on $X$
corresponding to $E$. There are Poincar\'e vector
bundles on $X\times M_\xi$; two such vector bundles differ by tensoring with
a line bundle pulled back from $M_\xi$. Fix a Poincar\'e vector bundle ${\mathcal E}$
on $X\times M_\xi$ as in \eqref{e0}. For any point $x\, \in\, X$, the vector bundle
${\mathcal E}\vert_{\{x\}\times M_\xi}$ on $M_\xi$ will be denoted by ${\mathcal E}_x$.
It is known that ${\mathcal E}_x$ is stable \cite[p.~174, Proposition~2.1]{LN}. If
$x$ and $y$ are two distinct points of $X$, then ${\mathcal E}_x$ is not isomorphic
to ${\mathcal E}_y$ \cite[p.~174, Theorem]{LN}. On the other hand, the infinitesimal
deformation map
$$
T_xX\, \longrightarrow\, H^1(M_\xi,\, \text{End}({\mathcal E}_x))
$$
is an isomorphism \cite[p.~392, Theorem~2]{NR}. Therefore, $X$ is a connected component
of the moduli space of stable sheaves on $M_\xi$ with numerical type that of
${\mathcal E}_x$. It may be noted that all sheaves on $M_\xi$ lying in this component of
the moduli space are actually locally free.

Fix a point $x\, \in\, X$. We will construct a family of co-Higgs fields on 
${\mathcal E}_x$ parametrized by the tangent line $T_xX\, \subset\, TX$.

The fiber ${\mathcal O}_X(x)_x$ over $x$
of the line bundle ${\mathcal O}_X(x)$ is identified with $T_xX$
using the Poincar\'e adjunction formula. More precisely, for any holomorphic
coordinate $z$ on $X$ defined around $x$ with $z(x)\,=\, 0$, the evaluation at $x$ of
the section $z\frac{\partial}{\partial z}$ of $TX\otimes{\mathcal O}_X(-x)$ is
independent of the choice of the function $z$. The isomorphism between ${\mathcal
O}_X(x)$ and $T_xX$ is given by this element of $(TX\otimes{\mathcal O}_X(-x))_x\,=
\, T_xX\otimes{\mathcal O}_X(-x)_x$. Let
$$
p\, :\, X\times M_\xi\,\longrightarrow\, X\ \ \text{ and }\ \
q\, :\, X\times M_\xi\,\longrightarrow\, M_\xi 
$$
be the natural projections. Consider the following short exact sequence of
sheaves on $X\times M_\xi$:
\begin{equation}\label{e1}
0\,\longrightarrow\,{\rm End}({\mathcal E})\,\longrightarrow\,{\rm End}({\mathcal E})
\otimes p^*{\mathcal O}_X(x) \,\longrightarrow\,{\rm End}({\mathcal E})_x\otimes
{\mathcal O}_X(x)_x\,=\, {\rm End}({\mathcal E})_x\otimes T_xX\,\longrightarrow\, 0\, ,
\end{equation}
where ${\rm End}({\mathcal E})_x\,=\, {\rm End}({\mathcal E}_x)$ is supported on
$\{x\}\times M_\xi$, and
$T_xX$ (respectively, ${\mathcal O}_X(x)_x$) denotes the trivial line bundle
on $\{x\}\times M_\xi$ with fiber $T_xX$ (respectively, ${\mathcal O}_X(x)_x$). Let
\begin{equation}\label{e2}
R^0q_* ({\rm End}({\mathcal E})_x\otimes T_xX)\,=\, {\rm End}({\mathcal E})_x\otimes T_xX
\,\stackrel{\gamma'}{\longrightarrow}\, R^1q_*({\rm End}({\mathcal E}))
\end{equation}
be the homomorphism in the long exact sequence of direct images associated to \eqref{e1}
for the projection $q$.

We have ${\rm End}({\mathcal E})\,=\, {\rm ad}({\mathcal E})\oplus {\mathcal O}_{X\times 
M_\xi}$, where ${\rm ad}({\mathcal E})\, \subset\, {\rm End}({\mathcal E})$ is the 
subbundle of corank one defined by the sheaf of trace zero endomorphisms, and the 
homomorphism ${\mathcal O}_{X\times M_\xi}\, \hookrightarrow\, {\rm End}({\mathcal E})$ 
is given by the scalar multiplications of ${\mathcal E}$. Therefore, we have
\begin{equation}\label{projeq}
R^1q_*({\rm End}({\mathcal E}))\,=\, R^1q_*({\rm ad}({\mathcal E}))\oplus
R^1q_*{\mathcal O}_{X\times M_\xi}\, .
\end{equation}
On the other hand, we have
$$
R^1q_*({\rm ad}({\mathcal E}))\,=\, TM_\xi\, ,
$$
because the fiber of $TM_\xi$ over any vector bundle $V\, \in\, M_\xi$ is
$H^1(X,\, \text{ad}(V))$. Therefore, from \eqref{projeq} we get a surjective homomorphism
\begin{equation}\label{projeq2}
R^1q_*({\rm End}({\mathcal E}))\,\longrightarrow\, TM_\xi\, \longrightarrow\, 0\, .
\end{equation}
Let ${\rm End}({\mathcal E})_x\otimes T_xX
\,\stackrel{\gamma'}{\longrightarrow}\, R^1q_*({\rm End}({\mathcal E}))
,\longrightarrow\, TM_\xi$ be the
composition of the
homomorphism $\gamma'$ in \eqref{e2} with the homomorphism in \eqref{projeq2}.
This composition produces a homomorphism
\begin{equation}\label{e3}
\gamma\, :\, T_xX\, \longrightarrow\, H^0(M_\xi,\, {\rm End}({\mathcal E})^*_x
\otimes TM_\xi)\,=\, H^0(M_\xi,\, {\rm End}({\mathcal E}_x)\otimes TM_\xi)\, .
\end{equation}
In other words, $\gamma(v)$ is a co-Higgs field on ${\mathcal E}_x$ for all $v\,\in\,
T_xX$.

\subsection{Non-integrability}

\begin{proposition}\label{prop1}
\mbox{}
\begin{enumerate}
\item The homomorphism $\gamma$ in \eqref{e3} is injective.

\item For any nonzero vector $v\, \in\, T_xX$, the co-Higgs field $\gamma(v)$ is not
integrable.
\end{enumerate}
\end{proposition}

\begin{proof}
The dimension of $T_xX$ in \eqref{e3} is one. Hence if the homomorphism $\gamma$ from 
$T_xX$ is not injective, then we have $\gamma\,=\, 0$. Therefore, the first statement in 
the proposition follows from the second statement. We will prove the second statement.

Lemma \ref{lem1} says that there are vector bundles $E\, \in \, M_\xi$ on $X$
such that
\begin{equation}\label{e4}
\dim H^0(X,\, \text{End}(E)\otimes{\mathcal O}_X(x))\,=\, 1\, .
\end{equation}
Fix such a vector bundle $E\, \in \, M_\xi$. Consider the natural short exact sequence
\begin{equation}\label{ne}
0\,\longrightarrow\,{\rm End}(E)\,\longrightarrow\,{\rm End}(E)
\otimes {\mathcal O}_X(x) \,\longrightarrow\,{\rm End}(E)_x\otimes
{\mathcal O}_X(x)_x\,=\, {\rm End}(E)_x\otimes T_xX\,\longrightarrow\, 0\, ;
\end{equation}
note that it coincides with the restriction of the exact sequence in \eqref{e1} to
$X\times\{E\}\, \subset\, X\times M_\xi$. Let
\begin{equation}\label{e5}
\begin{matrix}
0 & \longrightarrow & H^0(X,\, \text{End}(E)) & \stackrel{b}{\longrightarrow} &
H^0(X,\, \text{End}(E)\otimes{\mathcal O}_X(x)) \\
& \stackrel{s}{\longrightarrow} & {\rm End}(E)_x\otimes T_xX &
\stackrel{h}{\longrightarrow} &H^1(X,\, \text{End}(E)) & = & T_E M_\xi\oplus
H^1(X,\, {\mathcal O}_X)
\end{matrix}
\end{equation}
be the exact sequence of cohomologies associated to \eqref{e4}. Since $E$ is stable,
we have $H^0(X,\, \text{End}(E))\,=\,{\mathbb C}\cdot \text{Id}_E$. Therefore, from
\eqref{e4} it follows that the homomorphism $b$ in \eqref{e5} is an isomorphism.
Hence $s\,=\, 0$, which implies that $h$ in \eqref{e5} is injective.

The exact sequence in \eqref{ne} is a direct sum of the following two short exact
sequences
$$
0\,\longrightarrow\, {\mathcal O}_X\,\longrightarrow\,
{\mathcal O}_X(x) \,\longrightarrow\,
{\mathcal O}_X(x)_x\,=\, T_xX\,\longrightarrow\, 0
$$
and
$$
0\,\longrightarrow\,{\rm ad}(E)\,\longrightarrow\,{\rm ad}(E)
\otimes {\mathcal O}_X(x) \,\longrightarrow\,{\rm ad}(E)_x\otimes
{\mathcal O}_X(x)_x\,=\, {\rm ad}(E)_x\otimes T_xX\,\longrightarrow\, 0\, ,
$$
where $\text{ad}(E)\, \subset\,{\rm End}(E)$ is the
subbundle of co-rank one defined by the sheaf of trace-zero endomorphisms. Let
$$
h'\, :\, 
{\rm ad}(E)_x\otimes T_xX \,\longrightarrow, H^1(X,\, \text{ad}(E))\, =\,
T_E M_\xi
$$
be the homomorphism in the long exact sequence of cohomologies associated to the
second short exact sequence. Since $h$ \eqref{e5} is injective, we conclude that
$h'$ is also injective.
Therefore, the homomorphism of second exterior products induced by $h'$
$$
\bigwedge\nolimits^2 h' \, :\, \bigwedge\nolimits^2{\rm ad}(E)_x\otimes (T_xX)^{\otimes 2}\,=\,
\bigwedge\nolimits^2({\rm ad}(E)_x\otimes T_xX)\, \longrightarrow\,
\bigwedge\nolimits^2 H^1(X,\, \text{ad}(E))\,=\, \bigwedge\nolimits^2 T_E M_\xi
$$
is also injective.

Consider the Lie bracket homomorphism $\bigwedge\nolimits^2 {\rm End}(E_x)\, \longrightarrow\,
{\rm End}(E_x)$, defined by $A\bigwedge B \, \longmapsto\, \frac{1}{2}(AB- BA)$. Let
$$
\eta\, :\, {\rm End}(E_x) \,=\, {\rm End}(E_x)^* \, \longrightarrow\,
(\bigwedge\nolimits^2 {\rm End}(E_x))^*\,=\, \bigwedge\nolimits^2 {\rm End}(E_x)
$$
be the dual of this Lie bracket homomorphism. Let
$$
\eta'\, :\, {\rm End}(E_x) \, \longrightarrow\, \bigwedge\nolimits^2 {\rm ad}(E_x)
$$
be the composition of $\eta$ with the projection
$\bigwedge\nolimits^2 {\rm End}(E_x) \, \longrightarrow\, \bigwedge\nolimits^2 {\rm ad}(E_x)$
induced by the natural projection ${\rm End}(E_x) \, \longrightarrow\,{\rm ad}(E_x)$.

For any $v\, \in\, T_xX$, the element
$$
(\gamma(v)\bigwedge \gamma(v))(E)\, \in\, \text{End}(E_x)\otimes (\bigwedge\nolimits^2 T_E M_\xi)
$$
$$
=\, \text{End}(E_x)^*\otimes (\bigwedge\nolimits^2 T_E M_\xi)
\,=\,
\text{Hom}(\text{End}(E_x),\, \bigwedge\nolimits^2 T_E M_\xi)
$$
coincides with the homomorphism $\text{End}(E_x)\, \longrightarrow\,
\bigwedge\nolimits^2 T_E M_\xi$ defined by $$w\, \longmapsto\, (\bigwedge\nolimits^2 h')(\eta'(w)
\otimes v^{\otimes 2})\, ,\ \ w\, \in\, \text{End}(E_x)\, .$$ Now from the
injectivity of $\bigwedge\nolimits^2 h'$ it follows immediately that
$$
(\gamma(v)\bigwedge \gamma(v))(E)\,\not=\, 0
$$
if $v\, \not=\, 0$. Therefore, the co-Higgs field $\gamma(v)$ is not integrable
for all $v\, \not=\, 0$.
\end{proof}

\begin{lemma}\label{lem1}
Fix a point $x\,\in\, X$. There is a nonempty Zariski open subset $U_x\, \subset\,
M_\xi$ such that for all $E\, \in\, U_x$,
$$
\dim H^0(X,\, {\rm End}(E)\otimes{\mathcal O}_X(x))\,=\, 1\, .
$$
\end{lemma}

\begin{proof}
Write
\begin{equation}\label{m0}
{\rm degree}(\xi)\,=\, d+rm_0\, ,
\end{equation}
where $d$ and $m_0$ are integers
with $1\, \leq\, d\, <\, r$. We will first construct a vector bundle $V_r$ on $X$
of rank $r$ and degree $d_0$ such that 
\begin{equation}\label{Vr}
\dim H^0(X,\, \text{End}(V_r)\otimes{\mathcal O}_X(x))\,=\, 1\, .
\end{equation}
The locus in $\text{Pic}^{g-1}(X)$ of line bundles $L$
with $H^0(X,\, L)\, \not=\, 0$ is the theta divisor. Therefore,
for a general line bundle $L$ on $X$ with $\text{degree}(L)\,=\, g-1$,
we have $H^0(X,\, L)\, =\, 0$. Given that $g\, \geq\, 3$, this implies that for a general
line bundle $L$ in $\text{Pic}^{1}(X)$ or $\text{Pic}^{2}(X)$,
we have $H^0(X,\, L)\, =\, 0$. Consequently, there are holomorphic line bundles
$$
\{L_1\, , \cdots\, , L_d\, , L_{d+1}\, , \cdots\, , L_r\}
$$
on $X$ such that
\begin{enumerate}
\item $\text{degree}(L_i)\, =\, 1$ for all $1\, \leq\, i \, \leq\, d$,

\item $\text{degree}(L_i)\, =\, 0$ for all $d+1\, \leq\, i \, \leq\,r$

\item $H^0(X, \text{Hom}(L_i,\, L_j)\otimes{\mathcal O}_X(x))\,=\, 0$ for all $(i,\, j)
\,\in\, \{1\, , \cdots\, ,r\}\times \{1\, , \cdots\, ,r\}$ with $i\, \not=\, j$.
\end{enumerate}
We will now inductively construct holomorphic vector bundles $V_i$ of rank $i$,
$0\,\leq\, i\, \leq\, r$.

Set $V_0\,=\, 0$ and $V_1\, =\, L_1$. For each $2\,\leq\, i\, \leq\, r$, the vector
bundle $V_i$ fits in a short exact sequence of holomorphic vector bundles
\begin{equation}\label{e7}
0\, \longrightarrow\, V_{i-1} \, \longrightarrow\, V_i\, \longrightarrow\, L_i
\, \longrightarrow\, 0
\end{equation}
such that the corresponding extension class
$$
\mu_i\, \in\, H^1(X,\, \text{Hom}(L_i,\, V_{i-1}))
$$
satisfies the following condition: consider the exact sequence
$$
0\, \longrightarrow\, \text{Hom}(L_i,\, V_{i-2}) \, \longrightarrow\, \text{Hom}(L_i,\,
V_{i-1})\, \longrightarrow\, \text{Hom}(L_i,\, L_{i-1}) \, \longrightarrow\, 0\, ;
$$
it produces a surjective homomorphism
$$
h_i\, :\, H^1(X,\, \text{Hom}(L_i,\, V_{i-1}))\,\longrightarrow\,
H^1(X,\, \text{Hom}(L_i,\, L_{i-1})) \,\longrightarrow\, 0\, .
$$
The condition that $\mu_i$ is required to satisfy states as
\begin{equation}\label{mu}
h_i(\mu_i)\, \not=\, 0\, ;
\end{equation}
note that $H^1(X,\, \text{Hom}(L_i,\, L_{i-1}))\, \not=\, 0$ because by
Riemann Roch,
$$
\chi(\text{Hom}(L_i,\, L_{i-1}))\,=\, \text{degree}(\text{Hom}(L_i,\, L_{i-1}))
-g+1 \, \leq \, 1-g+1 \, <\, 0\, .
$$

It should be clarified that the above conditions do not determine $V_i$ uniquely.
We take $\{V_i\}_{i=0}^r$ to be a collection of vector bundles satisfying the
above conditions.

We will prove that \eqref{Vr} holds.

To prove \eqref{Vr}, consider the filtration
\begin{equation}\label{e6}
0\,=\, V_0\, \subset\, V_1\, \subset\, \cdots \, \subset\, V_{r-1}\, \subset\, V_r
\end{equation}
of $V_r$ by holomorphic subbundles obtained from \eqref{e7}. Take any
$$
T\, \in\, H^0(X,\, \text{End}(V_r)\otimes{\mathcal O}_X(x))\,.
$$
We will first prove that $T$ preserves the filtration in \eqref{e6}, meaning
\begin{equation}\label{T}
T(V_i)\, \subset\, V_i\otimes{\mathcal O}_X(x) 
\end{equation}
for all $i$.

Fix any $i\, \in\, \{1\, , \cdots \, ,r\}$, and consider the short exact sequence
\begin{equation}\label{f1}
\begin{matrix}
0 & \longrightarrow & \text{Hom}(V_{k+1}/V_k,\, L_j)\otimes{\mathcal O}_X(x)
 & \longrightarrow & \text{Hom}(V_{k+1},\, L_j)\otimes{\mathcal O}_X(x)\\
& \longrightarrow & \text{Hom}(V_k,\, L_j)\otimes{\mathcal O}_X(x)
&\longrightarrow 0
\end{matrix}
\end{equation}
obtained from \eqref{e7} by tensoring with
${\mathcal O}_X(x)$, where $j\, >\, i$ and $k\, <\, i$. Since
$$
H^0(X,\, \text{Hom}(V_{k+1}/V_k,\, L_j)\otimes{\mathcal O}_X(x))\,=\,
H^0(X,\, \text{Hom}(L_{k+1},\, L_j)\otimes{\mathcal O}_X(x))\,=\, 0
$$
(see the third condition on $\{L_i\}_{i=1}^r$), from the long exact sequence of
cohomologies associated to \eqref{f1} we conclude that
$$
H^0(X,\, \text{Hom}(V_{k+1},\, L_j)\otimes{\mathcal O}_X(x)) \,=\, 0
$$
if $H^0(X,\, \text{Hom}(V_{k},\, L_j)\otimes{\mathcal O}_X(x)) \,=\, 0$. Hence
using induction on $k$ it follows that
\begin{equation}\label{f2}
H^0(X,\, \text{Hom}(V_i,\, V_j/V_{j-1})\otimes{\mathcal O}_X(x))\,=\, 
H^0(X,\, \text{Hom}(V_i,\, L_j)\otimes{\mathcal O}_X(x)) \,=\, 0
\end{equation}
if $j\, >\, i$. Now from the exact sequence obtained from \eqref{e7} by tensoring with
${\mathcal O}_X(x)$
$$
\begin{matrix}
0 &\longrightarrow & \text{Hom}(V_i,\, V_j/V_{j-1}) \otimes{\mathcal O}_X(x)
&\longrightarrow & \text{Hom}(V_i,\, V_r/V_{j-1})\otimes{\mathcal O}_X(x)\\
& \longrightarrow & \text{Hom}(V_i,\, V_r/V_j)\otimes{\mathcal O}_X(x) 
&\longrightarrow & 0\, ,
\end{matrix}
$$
where $j \, >\, i$, it follows that
$$
H^0(X,\, \text{Hom}(V_i,\, V_r/V_{j-1})\otimes{\mathcal O}_X(x))\,=\, 0
$$
if $H^0(X,\, \text{Hom}(V_i,\, V_r/V_{j})\otimes{\mathcal O}_X(x))\,=\, 0$, because
\eqref{f2} holds. Therefore, using induction on $j$ it follows that
$$
H^0(X,\, \text{Hom}(V_i,\, V_r/V_i)\otimes{\mathcal O}_X(x))\,=\, 0\, ,
$$
which means that \eqref{T} holds.

Since \eqref{T} holds, the homomorphism $T$ induces a homomorphism
$$
T_i\, :\, L_i\, :=\,V_i/V_{i-1} \,\longrightarrow\, V_i/V_{i-1}
\otimes{\mathcal O}_X(x)\,=\, L_i\otimes{\mathcal O}_X(x)
$$
for each $i$. Now, $H^0(X, \, {\mathcal O}_X(x))\,=\,\mathbb C$ (for this it
is enough that $g\, \geq\, 1$), and hence it follows that
\begin{equation}\label{f3}
T_i\,=\, \lambda_i\cdot \text{Id}_{L_i}\, ,
\end{equation}
where $\lambda_i\, \in\,\mathbb C$.

As $H^0(X, \, \text{Hom}(V_i/V_{i-1},\, V_j/V_{j-1})\otimes{\mathcal O}_X(x))
\,=\, H^0(X, \, \text{Hom}(L_i,\, L_j)\otimes{\mathcal O}_X(x))\,=\, 0$
if $j\, <\, i$, it follows that there is no nonzero homomorphism
$S\, :\, V_r\, \longrightarrow\, V_r\otimes{\mathcal O}_X(x)$ over $X$ which
is nilpotent with respect to the filtration in \eqref{e6}, meaning
$S(V_i)\,\subset\, V_{i-1}\otimes{\mathcal O}_X(x)$ for all $i\, \geq\, 1$.

If $\lambda_1\,=\, \cdots \,=\, \lambda_r$ (constructed in \eqref{f3}), then 
$T-\lambda_1\cdot \text{Id}_{V_r}$ is nilpotent with respect to the filtration in 
\eqref{e6}. From the above observation that there are no nonzero nilpotent homomorphisms
it would then follow that $T\,=\, \lambda_1\cdot \text{Id}_{V_r}$. Therefore, to
prove \eqref{Vr} it suffices to show that
\begin{equation}\label{s2}
\lambda_1\,=\, \cdots \,=\, \lambda_r\, .
\end{equation}

Assume that \eqref{s2} fails. Let $k$ be the smallest integer such that
$\lambda_k \, \not=\, \lambda_1$. Now consider the restriction
$$
T'\, :=\, T\vert_{V_k}\, : \, V_k\, \longrightarrow\, V_k\otimes{\mathcal O}_X(x)\, .
$$
Let $L\, \subset\, V_k$ be the line subbundle generated by
$\text{kernel}(T'-\lambda_k\cdot\text{Id}_{V_k})$. The restriction to $L$ of
the projection $V_k\, \longrightarrow\, V_k/V_{k-1}\,=\, L_k$ is an isomorphism.
So $L$ provides a splitting of the short exact sequence obtained by setting
$i\,=\, k$ in \eqref{e7}. But this contradicts the assumption that the extension 
class $\mu_k$ is nonzero (see \eqref{mu}). Therefore, we conclude that \eqref{s2}
holds. As noted before, this proves \eqref{Vr}.

Let $\mathcal M$ be the moduli stack of vector bundles $W$ on $X$ such that
$\text{rank}(W)\,=\,r$ and $\text{degree}(W)\,=\, d$ (see \eqref{m0}). Let
${\mathcal C}\, \subset\, \mathcal M$ be the locus of all $W$ such that
$$\dim H^0(X,\, \text{End}(W)\otimes{\mathcal O}_X(x))\,=\,1\, .$$ This ${\mathcal C}$ is
open by semi-continuity and it is nonempty because $V_n\,\in\, {\mathcal C}$. Let
${\mathcal S}\, \subset\, \mathcal M$ be the stable locus which is also
nonempty and open \cite[p.~635, Theorem~2.8(B)]{Ma}. Finally, ${\mathcal C}$ and
${\mathcal S}$ intersect because $\mathcal M$ is irreducible \cite[p.~396,
Proposition~3.4]{BL}, \cite[p.~394, Proposition~2.6(e)]{BL} (see also \cite{DS}).

Let $F$ be a stable vector bundle on $X$ of rank $r$ and degree $d$
such that $$\dim H^0(X,\, \text{End}(F)\otimes{\mathcal O}_X(x))\,=\,1\, .$$ There is a
holomorphic line bundle $L$ of degree $m_0$ (see \eqref{m0}) such that
$F\otimes L\, \in\, M_\xi$. Since $\text{End}(F)\,=\, \text{End}(F\otimes L)$,
semi-continuity ensures the existence of $U_x$ in the statement of the lemma.
\end{proof}

\section{Computation of co-Higgs fields}

As before, take any point $x\, \in\, X$.

\begin{theorem}\label{thm1}
The homomorphism $\gamma$ in \eqref{e3} is an isomorphism.
\end{theorem}

\begin{proof}
The homomorphism $\gamma$ is injective by Proposition \ref{prop1}. We will
prove that
\begin{equation}\label{g1}
\dim H^0(M_\xi,\, {\rm End}({\mathcal E}_x)\otimes TM_\xi)\,=\, 1
\end{equation}
which would prove that $\gamma$ is surjective.

Let
$$
{\mathbb P}\, :=\, {\mathbb P}({\mathcal E}_x)\, \stackrel{\phi}{\longrightarrow}
\, M_\xi
$$
be the projective bundle of relative dimension $r-1$ that parametrizes all the
hyperplanes in the fibers of the vector bundle ${\mathcal E}_x$. Let
$$
T_\phi\, \longrightarrow\, {\mathbb P}
$$
be the relative tangent bundle for the projection $\phi$, so $T_\phi$ is the
kernel of the differential $d\phi\, :\, T{\mathbb P}\, \longrightarrow\, \phi^*TM_\xi$
of $\phi$.

\begin{lemma}\label{lem2}
There is a natural isomorphism
$$
H^0(M_\xi,\, {\rm End}({\mathcal E}_x)\otimes TM_\xi)\, \stackrel{\sim}{\longrightarrow}
\, H^0({\mathbb P},\, T_\phi\otimes \phi^*TM_\xi)\, .
$$
\end{lemma}

\begin{proof}
We have $\phi_* T_\phi\,=\, \text{ad}({\mathcal E}_x)$, so by the projection
formula,
$$
\phi_* (T_\phi\otimes \phi^*TM_\xi)\,=\, \text{ad}({\mathcal E}_x)\otimes TM_\xi\, .
$$
This implies that $H^0({\mathbb P},\, T_\phi\otimes \phi^*TM_\xi)\,=\,
H^0(M_\xi,\, {\rm ad}({\mathcal E}_x)\otimes TM_\xi)$. But
$$
H^0(M_\xi,\, {\rm End}({\mathcal E}_x)\otimes TM_\xi)\,=\,
H^0(M_\xi,\, {\rm ad}({\mathcal E}_x)\otimes TM_\xi)\oplus H^0(M_\xi,\,TM_\xi)\, ,
$$
and $H^0(M_\xi,\,TM_\xi)\,=\, 0$ \cite[p.~391, Theorem~1(a)]{NR}, \cite[p.~110,
Theorem~6.2]{H2}. Therefore, the lemma follows.
\end{proof}

Let ${\mathcal N}$ denote the moduli space of semistable vector bundles $E$ on $X$
of rank $r$ and $\bigwedge^r E\, =\, \xi\otimes {\mathcal O}_X(-x)$.

For any point $(E,\, H) \, \in\, {\mathbb P}$, we have a vector bundle $V$ on $X$
that fits in the short exact sequence
\begin{equation}\label{g3}
0\, \longrightarrow\, V\, \longrightarrow\,E\, \longrightarrow\, E_x/H
\, \longrightarrow\, 0\, .
\end{equation}
Note that $\bigwedge^r V\, =\, \xi\otimes {\mathcal O}_X(-x)$, however $V$ is not
semistable in general. Nevertheless, there is a nonempty Zariski open subset
$$
{\mathcal U}\, \subset\, \mathbb P
$$
satisfying the following four conditions:
\begin{enumerate}
\item the codimension of the complement ${\mathcal U}^c\, \subset\, \mathbb P$ is at
least three,

\item for every $(E,\, H) \, \in\, {\mathcal U}$, the corresponding vector bundle
$V$ constructed in \eqref{g3} is semistable,

\item for the resulting map
\begin{equation}\label{g4}
\psi\, :\, {\mathcal U}\, \longrightarrow\, {\mathcal N}\, ,
\end{equation}
the image ${\mathcal Y}\, :=\, \psi({\mathcal U})$ is Zariski open in ${\mathcal N}$ with
the codimension of the complement being ${\mathcal Y}^c\, \subset\, {\mathcal N}$
at least three, and

\item the map $\psi$ is a projective fibration over ${\mathcal Y}$.
\end{enumerate}
(See \cite{NR}.)

Consider the differential $d\psi\, :\, T {\mathcal U}\, \longrightarrow\, 
\psi^*{\mathcal N}$ of $\psi$ in \eqref{g4}. The kernel $$T_\psi\, :=\, 
\text{kernel}(d\psi) \, \subset\, T{\mathcal U}$$ is the relative tangent bundle on 
${\mathcal U}$.

Let $\widetilde{T}_x\,:=\, {\mathcal U}\times T_xX\, \longrightarrow\, \mathcal U$ be
the trivial line bundle on ${\mathcal U}$ with fiber $T_xX$. On $\mathcal U$, we have
\begin{equation}\label{s1}
T_\phi\,=\, \widetilde{T}_x\otimes T^*_\psi
\end{equation}
(see \cite[p.~265, (2.7)]{Bi}). Next we have
$$
R^0\psi_*(T^*_\psi\otimes \phi^*TM_\xi)\,=\, {\mathcal O}_{\mathcal Y}
$$
(see \cite[p.~266, Lemma~3.1]{Bi}). So,
$$
H^0({\mathcal U},\, T^*_\psi\otimes \phi^*TM_\xi)\,=\, \mathbb C\, .
$$
Now combining Lemma \ref{lem2} and \eqref{s1} it follows that
$$
H^0(M_\xi,\, {\rm End}({\mathcal E}_x)\otimes TM_\xi)\, =
\, T_xX\, .
$$
This proves \eqref{g1}.
\end{proof}

\begin{remark}\label{rem1}

The assumption that $g\, \geq\, 3$ was used in the proofs of Proposition \ref{prop1}
and Theorem \ref{thm1}; more precisely, this assumption is used in the proofs of
Lemma \ref{lem1} and Lemma \ref{lem2}. Well-known is the fact that a holomorphic
vector bundle on ${\mathbb C}{\mathbb P}^1$ decomposes into a direct sum of
holomorphic line bundles \cite{Gr}; hence, there are no strictly stable vector bundles on
${\mathbb C}{\mathbb P}^1$ of rank larger than $1$. From Atiyah's classification of vector bundles
on an elliptic curve $Y$ \cite{At}, we know that the moduli space $M_\xi$ for
$Y$ is a single point. What remains is the case of $g\,=\,2$. When
$r\,=\,g\,=\,2$, the moduli space $M_\xi$ has an explicit description
\cite{NR1}. However it is not clear whether Proposition \ref{prop1}
and Theorem \ref{thm1} hold in this special case.
\end{remark}

\section*{Acknowledgements}

We thank the referee for helpful comments. The first author acknowledges support of a J. 
C. Bose Fellowship. The second author acknowledges the support of a New Faculty 
Recruitment Grant from the University of Saskatchewan.

%%%%%%%%%%%%%%%%%%%%%%%%%%%%%%%%%%%%%%%%%%%%%%%%%%%%%%%%%%%%%%


\begin{thebibliography}{ZZZZZ}

\bibitem[At]{At} M. F. Atiyah, Vector bundles over an elliptic curve, {\it Proc. London Math. 
Soc.} {\bf 7} (1957), 414--452.

\bibitem[B1]{B1} E. Ballico and S. Huh, A note on co-Higgs bundles, arXiv:1606.01843.

\bibitem[B2]{B2} E. Ballico and S. Huh, $2$-nilpotent co-Higgs structures, arXiv:1606.02584.

\bibitem[BL]{BL} A. Beauville and Y. Laszlo, Conformal blocks and generalized theta
functions, {\it Comm. Math. Phys.} {\bf 164} (1994), 385--419. 

\bibitem[Bi]{Bi} I. Biswas, Infinitesimal deformations of the tangent bundle of a
moduli space of vector bundles over a curve, {\it Osaka Jour.
Math.} {\bf 43} (2006), 263--274.

\bibitem[BBGL]{BB} I. Biswas, U. Bruzzo, B. Gra\~na Otero, and A. Lo Giudice, 
Yang-Mills-Higgs connections on Calabi-Yau manifolds II, {\it Travaux Math.}
{\bf 24} (2016), 167--181.

\bibitem[Co]{Co} M, Corr\^ea, Rank two nilpotent co-Higgs sheaves on complex 
surfaces, \emph{Geom. Dedicata} {\bf 183} (2016), 25--31.

\bibitem[DS]{DS} V. G. Drinfelʹd and C. Simpson, $B$--structures on $G$--bundles and 
local triviality, {\it Math. Res. Lett.} {\bf 2} (1995), 823--829.

\bibitem[Gr]{Gr} A. Grothendieck, Sur la classification des fibr\'es holomorphes sur la 
sph\`ere de Riemann, {\it Amer. Jour. Math.} {\bf 79} (1957), 121--138.

\bibitem[FW]{FW} E. Frenkel and E. Witten, Geometric endoscopy and mirror symmetry, 
\emph{Commun. Number Theory Phys.} {\bf 2} (2008), 113--283.

\bibitem[Hi1]{H1} N. J. Hitchin. The self-duality equations on a Riemann surface, 
\emph{Proc. London Math. Soc.} {\bf 55} (1987), 59--126.

\bibitem[Hi2]{H2} N. J. Hitchin, Stable bundles and integrable systems, {\it Duke 
Math. Jour.} {\bf 54} (1987), 91--114.

\bibitem[Hi3]{H3} N. J. Hitchin, Generalized holomorphic bundles and the $B$-field 
action, \emph{Jour. Geom. Phys.} {\bf 61} (2011), 352--362.

\bibitem[Iy]{Iy} J. N. N. Iyer, Stability of tangent bundle on the moduli space of 
stable bundles on a curve, arXiv:1111.0196v5.

\bibitem[Hw]{Hw} J.-M. Hwang, Tangent vectors to Hecke curves on the moduli space of 
rank $2$ bundles over an algebraic curve, \emph{Duke Math. Jour.} {\bf 101} (2000),
179–187.

\bibitem[LN]{LN} H. Lange and P. E. Newstead, On Poincar\'e bundles of vector
bundles on curves, {\it Manuscripta Math.} {\bf 117} (2005), 173--181.

\bibitem[Ma]{Ma} M. Maruyama, Openness of a family of torsion free sheaves, {\it 
Jour. Math. Kyoto Univ.} {\bf 16} (1976), 627--637.

\bibitem[Mi]{Mi} Y. Miyaoka, Stable Higgs bundles with trivial Chern classes: 
several examples, \emph{Tr. Mat. Inst. Steklova} {\bf 264} (2009), Mnogomernaya 
Algebraicheskaya Geometriya, 129--136; translation in \emph{Proc. Steklov Inst. 
Math.} {\bf 264} (2009), no. 1, 123--130.

\bibitem[NR1]{NR1} M. S. Narasimhan and S. Ramanan, Moduli of vector bundles on a compact 
Riemann surface, {\it Ann. Math.} {\bf 89} (1969), 14--51.

\bibitem[NR2]{NR} M. S. Narasimhan and S. Ramanan, Deformations of the moduli space
of vector bundles over an algebraic curve, {\it Ann. Math.} {\bf 101} (1975), 391--417. 

\bibitem[Ra1]{R0} S. Rayan, \emph{Geometry of Co-Higgs Bundles}, D.Phil. Thesis, 
Oxford (2011).

\bibitem[Ra2]{R1} S. Rayan, Constructing co-Higgs bundles on ${\mathbb C}{\mathbb 
P}\sp 2$, \emph{Q. J. Math.} {\bf 65} (2014), 1437--1460.

\bibitem[Ra3]{R2} S. Rayan, Co-Higgs bundles on ${\mathbb P}\sp 1$, \emph{New York 
J. Math.} {\bf 19} (2013), 925--945.

\bibitem[Si]{Si} C. T. Simpson, Moduli of representations of the fundamental group 
of a smooth projective variety I, \emph{Inst. Hautes \'Etudes Sci. Publ. Math.} {\bf 
79} (1994), 47--129.

\bibitem[VC]{VC} A. Vicente-Colmenares, Moduli Spaces of semistable rank $2$ 
co-Higgs bundles over ${\mathbb P}\sp1\times{\mathbb P}\sp1$, arXiv:1604.01372.

\end{thebibliography}
\end{document}